\documentclass[12pt]{amsart}
\usepackage{amsmath,amsfonts,amssymb}
\usepackage{latexsym}
\usepackage{hyperref}

\theoremstyle{plain}
\newtheorem{theorem}{Theorem}
\newtheorem{corollary}[theorem]{Corollary}
\newtheorem{lemma}[theorem]{Lemma}
\newtheorem{proposition}[theorem]{Proposition}

\theoremstyle{definition}
\newtheorem{definition}[theorem]{Definition}

\newcommand{\Rn}{\mathbb{R}^n}
\newcommand{\essi}{\operatornamewithlimits{ess\,inf}}
\newcommand{\esss}{\operatornamewithlimits{ess\,sup}}

\newcommand{\al}{\alpha}
\newcommand{\bt}{\beta}
\newcommand{\dl}{\delta}

\newcommand{\lb}{\lambda}

\newcommand{\ve}{\varepsilon}
\newcommand{\gm}{\gamma}
\newcommand{\Gm}{\Gamma}
\newcommand{\sg}{\sigma}

\newcommand{\vi}{\varphi}

\newcommand{\rn}{\mathbb{R}^n}

\newcommand{\intl}{\int\limits}

\newcommand{\vs}{\vspace}
\newcommand{\me}{\mathrm{e}}

\title[Characterization of the variable exponent Bessel \ldots]{Characterization of the variable exponent Bessel potential spaces\\ via the Poisson semigroup}

\author[H. Rafeiro]{Humberto Rafeiro$^1$}
\address{Universidade do Algarve\\ Faculdade de Ci\^encias e Tecnologia\\ Campus de Gambelas, 8000-117 Faro -- PORTUGAL}
\thanks{$^1$ Supported by \textsl{Funda\c c\~ao para a Ci\^encia e a Tecnologia} (FCT),  Grant No. SFRH/BD/22977/2005, Portugal.}
\email{hrafeiro@ualg.pt}
\urladdr{http://w3.ualg.pt/~hrafeiro}
\author[S. Samko]{Stefan Samko$^2$}
\thanks{$^2$ Corresponding author.}
\address{Universidade do Algarve\\ Faculdade de Ci\^encias e Tecnologia\\ Campus de Gambelas, 8000-117 Faro -- PORTUGAL}
\email{ssamko@ualg.pt}
\urladdr{http://w3.ualg.pt/~ssamko}


\begin{document}
\begin{abstract}

Under the standard assumptions on the variable exponent  $p(x)$ (log- and
decay conditions), we give a characterization of the variable exponent Bessel
potential space $\mathfrak B^\alpha\left[L^{p(\cdot)}(\mathbb R^n)\right]$ in
terms of the rate of convergence of the Poisson semigroup $P_t$. We show that
the existence of the Riesz fractional derivative $\mathbb{D}^\al f$ in the
space $L^{p(\cdot)}(\rn)$ is equivalent to the existence of the limit
$\frac{1}{\ve^\al}(I-P_\ve)^\al f$. In the pre-limiting case $\sup_x
p(x)<\frac{n}{\al}$ we show that the Bessel potential space  is characterized
by the condition $\|(I-P_\ve)^\al f\|_{p(\cdot)}\leqq C \ve^\al$.

\end{abstract}
\maketitle

\section{Introduction}

The Bessel potential space $\mathfrak B^\alpha\left[L^{p(\cdot)}(\mathbb
R^n)\right],\ \alpha> 0,$  defined as the  range of the Bessel potential
operator over the variable exponent Lebesgue spaces  $L^{p(\cdot)}(\mathbb
R^n),$ was recently studied in \cite{alsamk}, under assumptions on $p(x)$
typical for \textsl{variable exponent analysis}, where in particular it was
shown that the space $\mathfrak B^\alpha\left[L^{p(\cdot)}(\mathbb
R^n)\right]$ may be characterized as the Sobolev space
\begin{equation}\label{inicio}
 L^{\alpha,p(\cdot)}(\mathbb R^n) = \left\{ f\in L^{p(\cdot)}: \;
\mathbb{D}^\alpha f \in L^{p(\cdot)} \right\},
\end{equation}
 with the Riesz fractional
derivative $\mathbb{D}^\alpha f $ realized  as a hypersingular integral, the
justification of the coincidence $\mathfrak
B^\alpha\left[L^{p(\cdot)}(\mathbb R^n)\right]=L^{\alpha,p(\cdot)}(\mathbb
R^n)$ being given in \cite{alsamk} in the ``under-limiting" case
$\sup_{x\in\mathbb{R}^n} p(x)<\frac{n}{\al}$. In \cite{alsamk}, in the case
of integer $\al$, it was also verified that  $\mathfrak
B^\alpha\left[L^{p(\cdot)}(\mathbb R^n)\right]$ coincides with the standard
Sobolev space, defined in terms of partial derivatives, the same having been
also checked in \cite{gurharnek}.

  In the case of constant $p$ it was also known that the Bessel potential space
$\mathfrak B^\alpha\left[L^{p(\cdot)}(\mathbb R^n)\right]$ may be
characterized in terms of the rate of convergence of identity approximations.
For instance, with the usage of the Poisson semigroup $P_t, t>0$, the space
$\mathfrak B^\alpha\left[L^{p}(\mathbb R^n)\right]$ is described as the
subspace of $L^{p}(\mathbb R^n)$ of functions $f$ for which there exists the
limit $\lim\limits_{t\to 0}\frac{1}{t^\al}(I-P_t)^\al f$, besides this
\begin{equation}\label{00}
\lim\limits_{\substack{t\to 0 \\ (L^p)}}\frac{1}{t^\al}(I-P_t)^\al f=\mathbb{D}^\al f
\end{equation}
 see for instance
Theorem B in \cite{hypsam}, where the simultaneous existence of the left- and
right-hand sides in \eqref{00} and their coincidence  was proved under the
assumption that $f$ and $ \mathbb{D}^\al f$ may belong to $L^r(\mathbb{R}^n)$
and $L^p(\mathbb{R}^n)$ with different $p$ and $r$. In the case $p=r$ this
was
 proved in \cite{520z} where the case of the Weierstrass
semigroup was also considered. Relations of  type \eqref{00} go back to
Westphal's formula \cite{704}
$$(-A)^\al f = \lim\limits_{h\to 0}\frac{1}{h^\al}(I-T_h)^\al f $$
for fractional powers  of the generator of a semigroup $T_h$ in a Banach
space. The latter in its turn has the origin in the Gr\"unwald-Letnikov
approach (\cite{gru,let}) to fractional derivatives of functions of one
variable,  under which the fractional derivative is defined as $
  \lim_{h\to 0+} \frac{\Delta_h^\alpha f}{h^\alpha},
$ where $\Delta_h^\alpha f$ is the difference of fractional order $\alpha>0$.

What is now called \textsl{variable exponent analysis} (VEA) was
intensively developed during the last two decades,  variable
exponent Lebesgue spaces $L^{p(\cdot)}(\mathbb{R}^n)$ being the
core of VEA. The progress in VEA was inspired both by difficult
open problems in this theory, and possible applications shown in
\cite{525}. Not going here into historical details, we refer to
original papers \cite{618a,332} and  surveying papers \cite{106b,
316b, 321n, 580bd}. As is known,  extension of various facts valid
for constant $p$ to the case of variable $p(x)$ encountered
enormous difficulties and required essential efforts from various
groups of researchers. Among the reasons we could remind that
variable exponent spaces are not invariant with respect to
translations and dilations, Young theorem for convolution
operators is no more valid, the Minkowsky integral inequality
proves to be a very rough inequality, etc.

Although expected, the validity of  \eqref{00} in the variable exponent
setting was not easy to justify, in particular because the apparatus of the
Wiener algebra of Fourier transforms of integrable function, based on  the
Young theorem, is not applicable. Another natural approach, based on Fourier
multipliers, extended in \cite{101ab} to the variable exponent setting, may
be already applicable, which is used in this paper. However,  because of the
specific behaviour of the Bessel functions appearing under the usage of the
Mikhlin-H\"ormander theorem, this approach also encountered essential
difficulties, see Section \ref{crucial} and Subsection \ref{added}.

The paper is arranged as follows.  In Section \ref{sec2} we provide some
necessary preliminaries. Section \ref{mainresults} contains formulations of
the main results of the paper, see Theorems \ref{main}, \ref{main2},
\ref{main1} and Corollary \ref{cor}. In Section \ref{crucial} we prove some
important technical lemmas and in  Section \ref{quatro} we give the proofs of
the main results. In particular in Subsection \ref{added} we show that some
specific functions are Fourier $p(x)$-multipliers, which required the most
efforts. The result on these Fourier multipliers is then used in Subsections
\ref{smain}-\ref{smain1} to obtain the characterization of Bessel potential
type spaces in terms of the rate of convergence of the Poisson semigroup.

\section{Preliminaries}\label{sec2}

 We refer to papers
\cite{618a, 332,  575a} and surveys \cite{106b, 316b, 580bd}   for
details on variable Lebesgue spaces, but  give some necessary
definitions. For a measurable function $p:\mathbb{R}^n \rightarrow
[1,\infty)$ we put $$p_+:= \esss_{x \in \mathbb{R}^n} p(x) \;\; \
\ \textrm{and}\;\;\ \
  p_-:=\essi_{x \in \mathbb{R}^n} p(x).$$
The variable exponent Lebesgue space $L^{p(\cdot)}(\mathbb{R}^n)$
is  the set of functions for which
$$\varrho_{p}(f):=\int_{\mathbb{R}^n} |f(x)|^{p(x)}\;d x<\infty .$$
In the sequel,  we suppose that $p(x)$ satisfies  one of the following standard conditions:
\begin{equation}\label{boundsnn}
1 \leqq p_{-}\leqq p(x)\leqq p_{+}<\infty,
\end{equation}
or
\begin{equation}\label{bounds}
1<p_-\leqq p(x) \leqq p_+<\infty.
\end{equation}
Equipped with the norm
$$\|f \|_{p(\cdot)}:=\inf \left\{\lambda>0:\varrho_{p}\left(
\frac{f}{\lambda} \right) \leqq 1 \right\},$$ this is a Banach
space.  By $p^\prime(x)$ we denote the conjugate exponent:
$\frac{1}{p(x)}+\frac{1}{p^\prime(x)} \equiv 1.$ We make use of
the well-known log-condition
\begin{equation}\label{llc}
|p(x)-p(y) |\leqq \frac{C}{-\ln \left(|x-y|\right)}, \qquad
|x-y|\leqq \frac{1}{2},\; x,y\in \mathbb{R}^n
\end{equation}
and  assume that there exists $\displaystyle p(\infty)=\lim_{x\rightarrow
\infty} p(x)$  and there holds the decay condition
\begin{equation}\label{decaycondition2}
|p(x)-p(\infty)|\leqq \frac{A}{\ln (2+|x|)},\qquad x\in
\mathbb{R}^n.
\end{equation}
\begin{definition}\label{def2} By $\mathcal{P}(\mathbb{R}^n)$ we denote the set of all bounded measurable functions $p:\mathbb{R}^n
\rightarrow [1,\infty)$ which satisfy assumptions \eqref{bounds},
\eqref{llc} and \eqref{decaycondition2}.
\end{definition}

\begin{definition}
By $\mathfrak M(\mathbb{R}^n)$ we denote the set of exponents
$p(\cdot):\mathbb{R}^n\to(1,\infty)$  such that the Hardy-Littlewood maximal
operator is bounded in the space $L^{p(\cdot)}(\Rn)$. As is  known
\cite{101ab}, $\mathcal{P}(\mathbb{R}^n)\subset \mathfrak M(\mathbb{R}^n)$.
\end{definition}

\subsection{Identity approximations}\label{identityappr}

Let $\phi\in L^1(\Rn)$ and  $\int_{\mathbb R^n} \phi(x)\;d x =1$. For each
$t>0$, we put $\phi_t:=t^{-n}\phi(xt^{-1})$. Following \cite{urifio07}, we
say that $\{\phi_t \}$ is a potential-type approximate identity, if it has
integrable radial majorant $$\sup_{|y|\geq|x|} |\phi(y)|\in L^1(\Rn).$$
Convergence of potential-type approximate identities in the setting of
variable exponent Lebesgue spaces $L^{p(\cdot)}$ was known from \cite{106}
under the assumption that the maximal operator is bounded. (An extension  to
some weighted spaces was given  in \cite{505z}). The following Proposition
\ref{urifio} proved in  \cite[Theorem 2.3]{urifio07}, does not use the
information about the maximal operator and allows to include the cases where
$p(x)$ may be equal to $1$.

\begin{proposition}\label{urifio} Let
a function $p: \mathbb{R}^n\to [1,\infty)$ satisfy conditions
\eqref{boundsnn}, \eqref{llc} and \eqref{decaycondition2}. If
$\{\phi_t\}$ is a potential-type approximate identity then

\vspace{3mm}

$\mathrm{i)} \hspace{3mm}\| \phi_t \ast f   \|_{p(\cdot)}  \leqq C \| f   \|_{p(\cdot)}, $

 \vspace{4mm} \noindent for all $t>0$, with $C>0$ not depending on $t$ and $f$, and

\vspace{4mm} $\mathrm{ii)}\hspace{2mm} \displaystyle \lim_{t \to
0} \| \phi_t \ast f -f \|_{p(\cdot)}=0 , \ \quad \ \ f \in
L^{p(\cdot)}(\mathbb{R}^n).$

\end{proposition}

\subsection{Fourier $p(x)$-multipliers}

 Let $m\in L^\infty(\mathbb R^n)$. We
define the operator $T_m$ by
\begin{equation}\label{multiplier}
\widehat {T_m f}(\xi)=m(\xi)\widehat f(\xi).
\end{equation}
When $T_m$ generates a bounded operator on $L^{p(\cdot)}(\mathbb R^n)$, we
say that $m$ is a \textsl{Fourier $p(\cdot)$-multiplier}. The following
Mikhlin-type multiplier theorem for variable Lebesgue spaces is known, see
\cite[Theorem 4.5]{koksam} where it was proved in a weighted setting; note
that a similar theorem in the form of H\"ormander criterion (for variable
exponents proved in \cite[Section 2.5]{uribe}) requires  to check the
behaviour of less number of derivatives (up to order
$\left[\frac{n}{2}\right]+1$), but leads to stronger restrictions on $p(x)$.

\begin{theorem}\label{mihhor1} Let a function $m(x)$ be continuous everywhere
in $\mathbb R^n$, except for probably the origin, have the mixed
distributional derivative $\frac{\partial^n m}{\partial x_1\partial x_2\cdots
\partial x_n}$ and the derivatives $D^\alpha m=\frac{\partial^{|\alpha|} m
}{\partial x_1^{\alpha_1}\partial x_2^{\alpha_2} \cdots \partial
x_m^{\alpha_m}}$, $\alpha=(\alpha_1,\ldots, \alpha_n)$ of orders
$|\alpha|=\alpha_1+\cdots+\alpha_n \leqq n-1$ continuous beyond the origin
and \begin{equation}\label{mih1} |x|^{|\alpha|} |D^\alpha m(x)|\leqq C,\quad
|\alpha|\leqq n \end{equation} where the constant $C > 0$ does not depend on
$x$.  If $p$ satisfies condition \eqref{bounds} and $p \in \mathfrak
M(\mathbb R^n)$, then $m$ is a Fourier $p(\cdot)$-multiplier in $L^{p(\cdot)}
(\mathbb R^n )$.
\end{theorem}

It is easily seen that  Mikhlin condition \eqref{mih1} for radial
functions $\mathcal{M}(r)=m(|x|)$ is reduced to
\begin{equation}\label{reduced}
\left|r^k \frac{d ^k}{d r^k}\mathcal{M}(r)\right|\leqq C<\infty, \quad \
k=0,1,\ldots, n.
\end{equation}
Note that \eqref{reduced} is equivalent to
\begin{equation}\label{reducedequiv}
\left|\left(r \frac{d}{d r}\right)^k\mathcal{M}(r)\right|\leqq C<\infty,
\quad \ k=0,1,\ldots, n,
\end{equation}
since $\left(r\frac{d}{d r}\right)^k=\sum\limits_{j=1}^k C_{k,j}r^j
\frac{d^j}{d r^j}$ with constant $C_{k,j}$, where $C_{k,1}=C_{k,k}=1$.

\begin{lemma}\label{dilation}
  Let a function $m$ satisfy Mikhlin's  condition \eqref{mih1}. Then the function
  $m_\varepsilon(x):=m(\varepsilon x)$ satisfies \eqref{mih1} uniformly in
$\ve$, with the same constant $C.$
  \end{lemma}

The proof is obvious since $ D^\alpha
m(\ve\cdot)(\xi)=\ve^{|\alpha|}(D^\alpha m)(\ve\xi).$

\vspace{3mm}

We need the following lemma on the identity approximations. Note that in
Lemma \ref{12} no information on the kernel is required: the only requirement
is that its Fourier transform satisfies the Mikhlin multiplier condition.

\begin{lemma}\label{12}
  Suppose that $m(x)$  satisfies  Mikhlin's condition \eqref{mih1}.  If
\[
  \lim_{\varepsilon \to 0} m(\varepsilon x)=1
\]
for almost all $x \in \mathbb R^n$ and  $p\in\mathcal{P}(\mathbb R^n)$, then
\begin{equation}\label{appoximation}
  \lim_{\varepsilon \to 0} \|T_{\varepsilon} f - f\|_{p(\cdot)}=0,
\end{equation}
for all $f \in L^{p(\cdot)}(\mathbb R^n)$, where $T_{\ve}$ is the operator generated by the
multiplier $m(\varepsilon x)$.
\end{lemma}

\begin{proof}
  The statement of the lemma is well known in the case of  constant
$p\in(1,\infty)$, see
  \cite[Lemma 12]{hypsam}, being valid in this case for an arbitrary
Fourier $p$-multiplier $m$.
  By Lemma
\ref{dilation} and Theorem \ref{mihhor1},  the family of operators
$\{T_{\ve}\}$ is uniformly
bounded in $L^{p(\cdot)}(\mathbb R^n)$. Therefore, it suffices to check \eqref{appoximation} on a
dense set in $L^{p(\cdot)}(\mathbb R^n)$, for instance for $f\in L^{p-}(\mathbb R^n)\cap
L^{p_+}(\mathbb R^n)$. Since $
  \|f\|_{p(\cdot)} \leqq \|f\|_{p_-} +\|f\|_{p_+},
$ for such functions $f$ we have
\[
  \left \|T_\varepsilon f - f \right\|_{p(\cdot)} \leqq
\|T_\varepsilon f- f\|_{p_-} +\|T_\varepsilon f- f\|_{p_+}
\]
from which the conclusion follows, in view of the validity of the theorem in case of constant $p$.
\end{proof}

\subsection{On finite differences}\label{findif}

By a finite difference of integer order $\ell$ and step $h \in \mathbb R$, in
this paper we always mean a non-centered  difference
\begin{equation}
   \Delta_h^\ell f  (x) = (I-\tau_h)^\ell f(x)\\
=\sum_{j=0}^\ell (-1)^j \binom{\ell}{j}f(x-jh)
\end{equation}
where $I$ is the identity operator and $\tau_h f(x)=f(x-h)$ is the
translation operator. We refer to \cite[Chapter 3]{hypsambook} and
\cite[Sections 25-26]{samkilbook} for more information on centered
or non-centered finite differences and their role  in fractional
calculus and the theory of hypersingular integrals.

The {\sl difference of fractional order $\alpha$} is defined as
\begin{equation}\label{fracdiff}
   \Delta_h^\alpha f  (x)= (I-\tau_h)^\alpha f(x)
=\sum_{j=0}^\infty (-1)^j \binom{\alpha}{j}f(x-jh), \qquad \alpha >0,
\end{equation}
where the series converges absolutely and uniformly for each $\alpha >0$ and
for every bounded function $f$, which follows from the fact that $
  \sum_{j=1}^\infty \left| \binom{\alpha}{j} \right | < \infty
$, see for instance \cite[Subsection 20.1]{samkilbook},  for  properties of
fractional order differences.

In a similar way there is introduced a {\sl generalized difference of
fractional order $\alpha$}, if one  replaces the translation operator
$\tau_h$ by any semigroup of operators. In this paper we make use of the
Poisson semigroup
\begin{equation}\label{1.27}
  P_t f(x)=\int_{\mathbb R^n} P(x-y,t)f(y)\;d y
\end{equation}
where
$
 \displaystyle P(x,t)= \frac{c_n\,t}{(|x|^2+t^2)^\frac{n+1}{2}},\quad
c_n=\Gamma\left( \frac{n+1}{2}\right)\Big \slash
\pi^\frac{n+1}{2},
$
so that
\begin{equation}\label{2.5}
(I-P_t)^\alpha f(x)=\sum_{k=0}^\infty (-1)^k \binom{\alpha}{k}
P_{kt}f(x).
\end{equation}
The Poisson kernel $P(x,t)$ is a potential-type approximate identity in
accordance with the definition of Section \ref{identityappr}. Therefore, by
Proposition \ref{urifio} the operators $P_tf$ are uniformly bounded in the
space $L^{p(\cdot)}(\mathbb{R}^n)$ under the assumptions of that Proposition
on $p(\cdot)$. Then
\begin{equation}\label{2.6}
  \|(I-P_t)^\alpha f \|_{p(\cdot)}\leqq C\,c(\alpha) \|f
\|_{p(\cdot)}, \qquad c(\alpha)=\sum_{k=0}^\infty \left|
\binom{\alpha}{k}\right|<\infty,
\end{equation}
where $C$ is the constant from the uniform estimate
$\|P_tf\|_{p(\cdot)} \leqq C \|f\|_{p(\cdot)}$, when
 $p: \mathbb{R}^n\to [1,\infty]$ satisfies conditions
\eqref{boundsnn}, \eqref{llc} and \eqref{decaycondition2}.

\subsection{Riesz potential operator and Riesz fractional derivative}
Recall that the \textsl{Riesz potential operator}, also known as \textsl{fractional integral operator}, is
given by
\begin{equation}\label{rieszoperator}
I^\alpha g(x):=\frac{1}{\gamma_{n}(\al)}\int_{\mathbb R^n}
\frac{g(y)}{|x-y|^{n-\al}}\;d y,  \quad 0<\alpha<n,
\end{equation}
with the normalizing constant  $\gamma_{n}(\al)=2^\alpha \pi^\frac{n}{2}
\frac{\Gamma\left(\frac{\al}{2}\right)}{\Gamma\left(\frac{n-\al}{2}\right)}$.
The hypersingular  integral
\begin{equation}\label{riesz_derivative}
\mathbb{D}^\alpha f(x):= \frac{1}{d_{n,\ell }(\alpha)}  \int_{\mathbb R^n}
\frac{\Delta_y^\ell  f (x)}{|y|^{n+\alpha}}\;d y ,
\end{equation}
where $\Delta_y^\ell  f(x)$ is a finite difference of order $\ell>
2\left[\frac{\al}{2}\right]$, is known as the \textsl{Riesz fractional
derivative}, see \cite[Chapter 3]{hypsambook} or \cite[Sections
26]{samkilbook} for the value of the normalizing constant  $d_{n,\ell
}(\alpha)$. It is known \cite{hypsambook, samkilbook} that operator
\eqref{riesz_derivative} is left inverse to the operator $I^\al$ within the
frameworks of $L^{p}$-spaces, which was extended to variable exponent spaces
$L^{p(\cdot)}(\mathbb R^n)$ in \cite{18c}.

Everywhere in the sequel, $\ell >\al$ and is even.

 When considered on functions in the range $I^\alpha (X)$ of the operator $I^\al$  over this or
that space $X$, the integral in (\ref{riesz_derivative})  is always
interpreted as the limit $\mathbb{D}^\alpha f:=\lim\limits_{\varepsilon \to
0} \mathbb{D}_\varepsilon^\alpha f$
 in the norm
of the space $X$,  of the truncated operators
\begin{equation}\label{truncatedhypersingular}
\mathbb{D}_\varepsilon^\alpha f(x) =\frac{1}{d_{n,\ell }(\alpha)} \int_{|y|>\varepsilon}
   \frac{\Delta_y^\ell  f  (x)}{|y|^{n+\alpha}} \;d y .
\end{equation}

\vs{1mm} The following proposition was proved in \cite[Theorem 5.5]{18c}
\begin{proposition}\label{alm}
Let $p\in \mathfrak M(\mathbb{R}^n)$ and $1<p_-(\mathbb R^n)\leqq p_+(\mathbb
R^n)<\frac{n}{\alpha}$. Then
$$\mathbb D^\alpha I^\alpha \varphi=\varphi,\qquad \varphi \in
L^{p(\cdot)}(\mathbb R^n),$$ where the hypersingular operator
 $\mathbb D^\alpha$ is understood as
convergent in $L^{p(\cdot)}$-norm.
\end{proposition}

We will also use the following result for variable exponent spaces, proved in
\cite[Theorems 3.2-3.3]{alsamk}. (Note that in \cite{alsamk} this statement
was formulated for $p$ satisfying  the log- and decay condition, but the
analysis of the proof shows that it uses only the fact that the maximal
operator is bounded).

\begin{proposition} \label{rieszrange1}
Let $p\in \mathfrak M(\mathbb{R}^n)$,  $1<p_-(\mathbb R^n)\leqq p_+(\mathbb
R^n)<\frac{n}{\alpha}$,  and let $f$ be a locally integrable function. Then
$f\in I^\alpha [ L^{p(\cdot)} ]$, if and only if $f\in L^{q(\cdot)}$, with
$\frac{1}{q(\cdot)} = \frac{1}{p(\cdot)} - \frac{\alpha}{n}$, and
\begin{equation} \label{uniformbound}
\|\mathbb{D}^\alpha_\varepsilon f \|_{p(\cdot)} \leqq C
\end{equation}
where $C$ does not depend on  $\varepsilon >0$.
\end{proposition}

It is known (\cite[p.70]{hypsambook}) that
\begin{equation}\label{1.10}  F (\mathbb D^\alpha_\varepsilon f)(x)=
\widehat{K_{\ell,\alpha}}(\varepsilon x) |x|^\alpha \widehat{f}(x), \ \quad f
\in C^\infty_0(\mathbb R^n),
\end{equation}
where $\widehat{K_{\ell,\alpha}}(x)$ is the Fourier transform of the function
$K_{\ell,\alpha}(x)$ with the property
\begin{equation}\label{1.12}
  K_{\ell,\alpha}(x)\in L^1(\mathbb R^n),\qquad \int_{\mathbb R^n}
K_{\ell,\alpha}(x)\;d x=1.
\end{equation}
The function $\widehat{K_{\ell,\alpha}}(x)$ is given explicitly by
\begin{equation}\label{1.13}
  \widehat{K_{\ell,\alpha}}(x)=\frac{(2\mathrm{i})^\ell}{d_{n,\ell}(\alpha)}
\int_{|y|>|x|} \frac{\sin^\ell (y_1)}{|y|^{n+\alpha}}\;d y.
\end{equation}

For brevity of notation, we will denote $\widehat{K_{\ell,\alpha}}(x)$ simply
as $w(|x|)$, therefore

\begin{equation}\label{ww}
w(|x|)=c\int_{|y|>|x|} \frac{\sin^\ell(y_1)}{|y|^{n+\alpha}}\;d y
=c\int_{|x|}^\infty \frac{V(\rho)}{\rho^{1+\alpha}}\, d \rho
\end{equation}
where
\begin{equation}\label{Vr}
V(\rho)=\int_{S^{n-1}} \sin^\ell (\rho \sigma_1)\;d \sigma \ \quad
\textrm{and} \ \quad c=\frac{(2\mathrm{i})^\ell}{d_{n,\ell}(\al)}.
\end{equation}

\begin{lemma} \label{lemma}
The following formula is valid
\begin{equation}\label{sphericalintegral}
V(\rho) = \lb + \sum_{i=0}^{\frac{\ell}{2}-1} C_i
\frac{J_{\nu-1}(\ell_i\rho)}{(\ell_i \rho)^{\nu-1}},
\end{equation}
where $\ell = 2,4,6,\ldots$, $J_{\nu-1}(r)$ is the Bessel function of the
first kind,  $\nu=\frac{n}{2}$, $\ell_i=\ell-2i$ and $\lb$ and $C_i $ are
constants:
\begin{equation}\label{values}
\lb = \frac{4\pi^\frac{n}{2}\Gm\left(\frac{\ell+1}{2}\right)
}{\ell\Gm\left(\frac{\ell}{2}\right)\Gm\left(\frac{n}{2}\right)}, \qquad C_i
=(-1)^{\frac{\ell}{2}-i}(2\pi)^\frac{n}{2}2^{1-\ell} \binom{\ell}{i}.
\end{equation}
\end{lemma}
\begin{proof} Formula \eqref{sphericalintegral} is a consequence of the
 Catalan formula
\begin{equation}\label{catfor}
  \int_{S^{n-1}} \sin^\ell
  (\rho \sigma_1)\;d \sigma = |S^{n-2}|
  \int_{-1}^1 \sin^\ell (\rho t) (1-t^2)^\frac{n-3}{2}\;d t,
\end{equation}
(see, for instance, \cite[p.13]{hypsambook}), the Fourier expansion
\begin{equation}\label{fourier}
  \sin^{\ell} (t) =\frac{1}{2^{\ell-1}} \sum_{i=0}^{\frac{\ell}{2}-1}
   (-1)^{\frac{\ell}{2}-i} \binom{\ell}{i} \cos \big((\ell-2i)t\big)+
   \frac{1}{2^{\ell}} \binom{\ell}{\ell/2}
\end{equation}
of the function $\sin^\ell (t)$ with even $\ell$ (see, e.g., \cite[Appendix
I.1.9]{prubrymar}), and the Poisson formula
\begin{equation}\label{intbes}
J_\nu(\rho)=\frac{(\rho/2)^\nu}{\Gamma\left( \frac{1}{2}\right)\Gamma\left(
\nu+\frac{1}{2}\right)} \int_{-1}^1 \cos (\rho t) \left(1-t^2 \right)^{\nu
-\frac{1}{2}}\;d t
\end{equation}
with $\Re \nu >-\frac{1}{2}$  for the Bessel function (see, e.g.,
\cite{lebedev}). The values in \eqref{values} are obtained by direct
calculations via properties of Gamma-function.
\end{proof}

\vs{3mm}Following  \cite{hypsam} (see also \cite[p. 214]{hypsambook}), we
 make use of  the
 functions
\begin{equation}\label{3.10}
  A(x)=\frac{(1-\me^{-|x|})^\alpha}{|x|^\alpha
w(|x|) } \qquad \mathrm{and}  \qquad B(x) = \frac{1}{A(x)},
\quad x\in \mathbb R^n,
\end{equation}
which will play a central role in this paper.

Since the functions  $A(x)$ and $B(x)$ are radial, we find it
convenient to also use the notation
\begin{equation}\label{conva}
\mathcal{A}(r)=\frac{(1-\me^{-r})^\alpha}{r^\alpha w(r)}  \qquad
\mathrm{and}  \qquad \mathcal{B}(r)=\frac{r^\alpha
w(r)}{(1-\me^{-r})^\alpha}.
\end{equation}

\subsection{Bessel potential operator}

The \textsl{Bessel potential} of order $\alpha >0$ of the density $\varphi$ is defined by
\begin{equation}
\mathfrak B^\alpha \varphi(x)=\int_{\mathbb R^n} G_\alpha(x-y) \varphi(y) \;d
y
\end{equation}
where the Fourier transform of the \textsl{Bessel kernel} $G_\alpha$ is given by
\[
\widehat{G_\alpha}(x)=(1+|x|^2)^{-\alpha/2},\quad x\in \mathbb R^n,\quad \alpha >0.
\]

\begin{definition}
We define the \textsl{variable exponent Bessel potential space}, sometimes
 also called \textsl{Liouville space of fractional smoothness},
 as the range of the Bessel potential operator
\[
\mathfrak B^\alpha\left[L^{p(\cdot)}(\mathbb R^n)\right]=\left\{f: \;
f=\mathfrak B^\alpha\varphi, \;\varphi\in L^{p(\cdot)}(\mathbb R^n)
\right\},\quad \alpha> 0.
\]

\end{definition}
The following characterization of the variable exponent Bessel potential
space
  via hypersingular integrals was given in \cite{alsamk}.

\begin{proposition}\label{besselcharacterization}
Let $0<\alpha<n$. If $1<p_-\leqq p_+<n/\alpha $ and $p(\cdot)$
satisfies conditions \eqref{llc} and \eqref{decaycondition2}, then
$\mathfrak B^\alpha\left(L^{p(\cdot)}\right)=L^{p(\cdot)}\cap
I^\alpha\left(L^{p(\cdot)}\right) $ with equivalent norms.
\end{proposition}

\section{Main results}\label{mainresults}

We  first prove that the functions $A(x), B(x)$ defined in
\eqref{3.10} are Fourier $p(\cdot)$-multipliers in
$L^{p(\cdot)}(\mathbb R^n)$ under suitable exponents $p(\cdot)$,
see Theorem \ref{amultiplier},
 which proved
to be the principal difficulty in extending the result in
\eqref{00} to variable exponents.
\begin{theorem}\label{amultiplier}
The function  $A$  is a Fourier $p(\cdot)$-multiplier when $p(\cdot) \in
\mathcal P (\mathbb R^n)$.
\end{theorem}

Then by means of Theorem \ref{amultiplier} we prove the following
statements.

\begin{theorem}\label{main} Let $f \in L^{p(\cdot)}(\mathbb
R^n),\; p(\cdot) \in \mathcal P(\mathbb R^n)$  and let $\mathbb
D^\alpha_\varepsilon f$ be the truncated hypersingular integral
\eqref{truncatedhypersingular}. The limits
\begin{equation}\label{principal} \lim_{\varepsilon\to
0+} \frac{1}{\varepsilon^\alpha}(I-P_\varepsilon)^\alpha f \quad \
\textrm{and} \ \quad \lim_{\varepsilon \to 0+} \mathbb
D^\alpha_\varepsilon f \end{equation} exist in
$L^{p(\cdot)}(\mathbb R^n)$ simultaneously and coincide with each
other.
\end{theorem}

\begin{corollary}\label{cor} Let $\al>0$ and $p\in \mathcal{P}(\rn)$.
The equivalent characterization of the space $L^{\alpha,p(\cdot)}(\mathbb
R^n)$ defined in \eqref{inicio}, is given by
\[L^{\alpha,p(\cdot)}(\mathbb R^n)=\left\{f\in L^{p(\cdot)}(\rn): \  \lim_{\varepsilon\to
0+} \frac{1}{\varepsilon^\alpha}(I-P_\varepsilon)^\alpha f\in
L^{p(\cdot)}(\rn)\right\}.\]

\end{corollary}

\begin{theorem}\label{main2}
Let $0< \alpha <n$, $1<p_- \leqq p_+ <n/ \alpha$, $p(\cdot)\in \mathcal
P(\mathbb R^n)$. A function $f\in L^{p(\cdot)}(\mathbb R^n) $ belongs to
$L^{\alpha, p(\cdot)}(\mathbb R^n)$ if and only if
\begin{equation}\label{vdnh}
\|(I-P_\ve)^\alpha f\|_{p(\cdot)}\leqq C \ve^\al,
\end{equation}
where $C$ does not depend on  $\ve$; condition \eqref{vdnh} being fulfilled,
it involves that $\mathbb{D}^\al f\in L^{p(\cdot)}(\rn)$ and \eqref{vdnh} is
also valid in the form $\|(I-P_\ve)^\alpha f\|_{p(\cdot)}\leqq C \|\mathbb{D}^\al
f\|_{p(\cdot)}\ve^\al$ where $C$ does not depend on $f$ and $\ve$.
\end{theorem}
\begin{theorem}\label{main1}
Let $0< \alpha <n$, $1<p_- \leqq p_+ <n/ \alpha$ and $p(\cdot)\in
\mathcal P(\mathbb R^n)$.
 The variable exponent Bessel potential space $\mathfrak B^\alpha(L^{p(\cdot)})$
 is the subspace in  $L^{p(\cdot)}(\mathbb R^n)$ of functions $f$ for which the limit
$\lim_{\varepsilon \to 0+}
\frac{1}{\ve^\al}(I-P_\varepsilon)^\alpha f $ exists.
\end{theorem}

\section{Crucial lemmata}\label{crucial}
We start with the following two simple  lemmas.
\begin{lemma}\label{inversion}
Let a function $\mathcal{M}(r)$ satisfy Mikhlin condition
\eqref{reduced} and $\inf_{x\in\mathbb{R}^1_+}|\mathcal{M}(r)|>0$.
Then the function $\frac{1}{\mathcal{M}(r)}$ satisfies the same
condition.
\end{lemma}
\begin{proof}
Statements of such a kind are well known, we just note that the assertion of
the lemma follows from \eqref{nth}-\eqref{ultima}.
\end{proof}
\begin{lemma}\label{nonvanish}
The function $\mathcal{B}(r)$ defined in \eqref{conva} is
non-vanishing: \\ $\inf_{r\in \mathbb{R}^1_+}|\mathcal{B}(r)|>0$.
\end{lemma}
\begin{proof}
The function $\mathcal{B}(r)$ is continuous in $(0,\infty)$ and
$|\mathcal{B}(r)|>0$ for all $r\in (0,\infty)$. Therefore, it suffices to
check that $\mathcal{B}(0)\ne 0$ and $\mathcal{B}(\infty)\ne 0$. From
\eqref{1.12} it follows that $\mathcal{B}(0)= 1$, while
$\mathcal{B}(\infty)=\frac{\lb}{\al}\ne 0$ is seen from the asymptotics
\eqref{pro} proved in Lemma \ref{lem0}.
\end{proof}

\begin{lemma}\label{inserted} The following formula is valid
\begin{multline}\label{3}
\intl_0^\infty f(t)t^\nu J_{\nu-1}(rt)\,d t=\\
 \quad=\frac{(-1)^m}{r^m}
\sum\limits_{k=1}^mc_{k,m}\intl_0^\infty f^{(k)}(t)t^{\nu+k-m}
J_{\nu+m-1}(rt)\,d t, \ \quad m\geqq 1,
\end{multline}
if
$$\left. f(t)t^\nu J_\nu(t)\right|_0^\infty=0$$
and
\[\left. f^{(k)}(t)t^{\nu+k-j} J_{\nu+j}(t)\right|_0^\infty=0, \ \ \
k=1,2,\ldots,j, \ \ \ j=1,2,\ldots, m-1,
\]
the latter appearing in the case  $m\geqq 2$.

\end{lemma}
\begin{proof} A relation of type \eqref{3} is known in the form
\begin{equation}\label{1}
\intl_0^\infty f(t)t^\nu J_{\nu-1}(t|x|)\,d
t=\frac{(-1)^m}{|x|^m}\intl_0^\infty f^{\langle m\rangle}(t)t^{\nu+m}
J_{\nu+m-1}(t|x|)\,d t,
\end{equation}
under the conditions
\begin{equation}\label{2}
\left. f^{\langle j\rangle}(t)t^{\nu+j}
J_{\nu+j}(t)\right|_0^\infty=0, \ \ \ j=0,1,2,\ldots, m-1.
\end{equation}
see  formula (8.133) in \cite{hypsam}, where it is denoted $f^{\langle
m\rangle}(t)=\left(\frac{1}{t}\frac{d}{d t}\right)^m f(t).$ Then \eqref{3}
follows from \eqref{1} if one observes that
\begin{equation}\label{267}
\left(\frac{1}{t}\frac{d}{d t}\right)^m f(t)= \sum\limits_{k=1}^m
c_{k,m}\frac{f^{(k)}(t)}{t^{2m-k}},
\end{equation}
where $c_{k,m}$ are constants (here and in the sequel, by $c, c_k,
c_{km}, c_{js}$ etc, we denote constants the exact values of which
are not important for us).
\end{proof}

The  following two  lemmas are crucial for our purposes.
\begin{lemma}\label{lem0}
The function $\mathcal{B}(r)$ has the following structure at infinity:
\begin{equation}\label{pro}
\mathcal{B}(r) =
\frac{\lb}{\al}+\frac{1}{r^\nu}\sum\limits_{i=0}^{\frac{\ell}{2}-1} \frac{C_i
}{\ell_i^\nu}J_{\nu-2}(\ell_ir) +
O\left(\frac{1}{r^{\nu+\frac{3}{2}}}\right), \quad r\to \infty
\end{equation}
where $\lb$ and $C_i$ are constants from \eqref{values}.
\end{lemma}
\begin{proof}
By \eqref{ww}  and \eqref{sphericalintegral}, we obtain
\begin{equation}\label{representation}
r^\al w(r)= \frac{\lb}{\al}+ r^\al \sum\limits_{i=0}^{\frac{\ell}{2}-1} C_i
\intl_r^\infty \frac{J_{\nu-1}(\ell_it)}{t^{1+\al} (\ell_i t)^{\nu-1}}d t.
\end{equation}
 By the well known differentiation formula
$\frac{J_\nu(t)}{t^{\nu-1}}=- \frac{d}{d
t}\left[\frac{J_{\nu-1}(t)}{t^{\nu-1}}\right]$ for the Bessel functions, via
integration by parts we arrive at the relation
\begin{equation}\label{byparts}
\intl_r^\infty\frac{J_\nu(t)}{t^\bt}\, d t=\frac{J_{\nu-1}(r)}{r^\bt}+
(\nu-\bt-1)\intl_r^\infty\frac{J_{\nu-1}(t)}{t^{\bt+1}}\, d t,
\end{equation}
for $r>0$ and $\bt>-\frac{1}{2}$. Applying repeatedly this formula
two times, we  transform \eqref{representation} to
\begin{multline*}
r^\al w(r)=\frac{\lb}{\al}+\sum\limits_{i=0}^{\frac{\ell}{2}-1}
\frac{C_i }{\ell_i^\nu}\left[\frac{J_{\nu-2}(\ell_ir)}{r^\nu}-
\frac{2+\al}{\ell_i} \frac{J_{\nu-3}(\ell_ir)}{r^{\nu+1}}\right]+
\\ + (\al+2)(\al+4)
\sum\limits_{i=0}^{\frac{\ell}{2}-1}\frac{r^\al}{\ell_i} \intl_r^\infty
\frac{J_{\nu-3}(\ell_it)}{t^{\al+\nu+2}}\,d t, \ \ \quad 0<r<\infty,
\end{multline*}
 whence \eqref{pro} follows, since $\mathcal{B}(r)=r^\al w(r)+O(\me^{-r})$ as $ r\to\infty$.
\end{proof}

\begin{lemma}\label{lem}
The derivatives $\mathcal{B}^{(k)}(r)$ have the following structure at
infinity:
\begin{equation}\label{pronew}
\mathcal{B}^{(k)}(r) =
\frac{c}{r^k}+\frac{1}{r^\nu}\sum\limits_{i=0}^{\frac{\ell}{2}-1} c_i
J_{\nu-2}(\ell_i r) + O\left(\frac{1}{r^{\nu+\frac{1}{2}}}\right), \ \ \quad
r\to \infty
\end{equation}
where $c$ and $c_i$ are constants.

\end{lemma}
\begin{proof}
By Leibniz' formula it suffices to show that the derivatives $[r^\al
w(r)]^{(k)}$ have the same asymptotics at infinity as in \eqref{pronew}. We
have
\begin{equation}\label{toverify}
\left[r^\al w(r)\right]^{(k)}=c_{k,0}r^{\al-k} w(r)
+\sum\limits_{j=1}^kc_{k,j} r^{\al+j-k}w^{(j)}(r)
\end{equation}
From \eqref{ww} we have $w^\prime(r)=\frac{V(r)}{r^{1+\al}}$, so
that
\begin{equation}\label{transformed}
\left[r^\al w(r)\right]^{(k)}=c_{k,0}r^{\al-k} w(r)
+\sum\limits_{j=0}^{k-1}c_{k,j+1} r^{\al+j+1-k}\frac{d^j}{d r^j}
\left(\frac{V(r)}{r^{1+\al}}\right)
\end{equation}
Hence
\begin{equation} \label{relation}
\left[r^\al w(r)\right]^{(k)}=c_{k,0}r^{\al-k} w(r)
+\sum\limits_{i=0}^{k-1}c_{k,i} r^{i-k}V^{(i)}(r).
\end{equation}
We make use of the relation
$$\left(\frac{d}{d r}\right)^i = \sum\limits_{s=0}^{[\frac{i}{2}]}
c_{i,s}r^{i-2s}\mathfrak D^{i-s}, \ \ \quad \mathfrak D=\frac{d}{rd r}$$
and transform \eqref{relation} to
$$\left[r^\al w(r)\right]^{(k)}=c_{k,0}r^{\al-k} w(r)
+\sum\limits_{s=0}^{k-1}c_{k,s} r^{2s-k}\mathfrak D^{s}V(r),
$$
keeping in mind formula \eqref{sphericalintegral}. Then by
\eqref{sphericalintegral} and the  formula $\mathfrak
D^s\left[\frac{J_\nu(r)}{r^\nu}\right]= (-1)^s
\frac{J_{\nu+s}(r)}{r^{\nu+s}}$, after easy transformations  we arrive at the
equality
$$
\left[r^\al w(r)\right]^{(k)}=cr^{\al-k} w(r) + c_1r^{-k}+
\sum\limits_{s=0}^{k-1}r^{-\nu-s}
\sum\limits_{i=0}^{\frac{\ell}{2}-1}c_{s,i}J_{\nu+k-s-2}(\ell_i r)
$$
for $0<r<\infty$. In view of \eqref{pro}, we arrive then at
\eqref{pronew}.
\end{proof}

\section{Proofs of the main result}\label{quatro}

\subsection{Proof of Theorem \ref{amultiplier}}\label{added}

$\left.\right.$

\vs{3mm} Since in the proof of Theorem \ref{amultiplier}, we check the
Mikhlin condition, in view of  Lemmas \ref{inversion} and \ref{nonvanish} it
suffices to prove this theorem only for $\mathcal{B}(x).$

We  need to deal with the behaviour of $\mathcal B(r)$ in different way
 near the origin and infinity. To this end, we make use of a unity
partition  $ 1\equiv \mu_1(r)+\mu_2(r)+\mu_3(r),\ \mu_i \in C^\infty,
i=1,2,3,$ where
\begin{equation}\label{particao}
  \mu_1(r)=     \left\{
\begin{array}{rl}
  1 & \text{if }  0 \leqq x < \varepsilon\\
   0 & \text{if } x \geqq \varepsilon + \delta \\
\end{array} \right., \qquad
\mu_3(r)=     \left\{
\begin{array}{rl}
  0 & \text{if }  0 \leqq x < N-\delta\\
   1 & \text{if } x \geqq N \\
\end{array} \right.
\end{equation}
with $ \mathrm{supp}\; \mu_1=[0,\varepsilon+\delta], \ \mathrm{supp}\;
\mu_3=[N-\delta,\infty)$, and represent $\mathcal{B}(r)$ as

\begin{equation}\label{splitting}
\begin{split}
  \mathcal{B}(r)&=\left(\frac{1-\me^{-r}}{r} \right)^{\!\!-\alpha}\!\! w(r)  \,
\mu_1(r)+\mathcal{B}(r) \, \mu_2(r)+(1-\me^{-r})^{-\alpha}\, r^\al w(r)\, \mu_3(r)\\
&=:\mathcal{B}_1(r)+\mathcal{B}_2(r)+\mathcal{B}_3(r).
\end{split}
\end{equation}

The function $\mathcal{B}_2(r)$ vanishing in the neighbourhoods of the origin
and infinity, is infinitely differentiable, so that it is a Fourier
multiplier in $L^{p(\cdot)}(\rn)$. Therefore, we only have to take care of
the multipliers $\mathcal{B}_1(r)$ and $\mathcal{B}_3(r)$ supported in
neighbourhoods of origin and infinity, respectively. They will be treated in
a different way. For $\mathcal{B}_1(r)$ we will apply the Mikhlin criterion
for the spaces $L^{p(\cdot)}(\rn)$, while the case of the multiplier
$\mathcal{B}_3(r)$ proved to be more difficult. In the case $n=1$ it is
easily covered by means of the Mikhlin criterion, while for $n\geqq 2 $ we use
another approach. Namely, we show that the kernel $a_3(|x|)$, corresponding
to the multiplier
\begin{equation}\label{labelast}
\mathcal{B}_3(r)-\mathcal{B}_3(\infty)=
\mu_3(r)[\mathcal{B}_3(r)-\mathcal{B}_3(\infty)],
\end{equation}
has an integrable radial nonincreasing majorant, which will mean that
$\mathcal{B}_3(r)$ is certainly a multiplier. However, this will require the
usage of  special facts on behaviour of the Bessel functions at infinity and
an information on some of integrals of Bessel functions.

The proof of Theorem \ref{amultiplier} follows from the study of  the
multipliers $\mathcal{B}_1(r)$ and $\mathcal{B}_3(r)$ made in Subsections
\ref{B1} and \ref{B3}.

\subsubsection{Proof for the case of the multiplier  $\mathcal{B}_1(r)$}\label{B1}

\begin{lemma}\label{amultiplieras} The function  $\mathcal{B}_1(r)$ satisfies
Mikhlin condition \eqref{reduced}.
\end{lemma}

\begin{proof} We have to check condition \eqref{reduced} only near the
origin.

 The function $g(r):=\frac{r}{1-\me^{-r}}$ with $g(0)\ne 0$ is non-vanishing
in any neighbourhood of the origin. Therefore, the function
$\left(\frac{r}{1-\me^{-r}}\right)^\alpha=\left[g(r)\right]^\al$ is
infinitely differentiable on any finite interval $[0,N]$ and thereby
satisfies conditions \eqref{reduced} on every neighbourhood of the origin.

  Thus, to estimate $r^k \frac{d ^k}{d r^k}\mathcal{B}_1(r)$, we only need to show the boundedness of $
\left|r^kw^{(k)}(r)\right| $ as $r\to 0$.  By the equivalence
\eqref{reduced}$\Longleftrightarrow$\eqref{reducedequiv}, we may estimate
$\left(r\frac{d}{d r}\right)^{j} w(r)$. Since $ w'(r)=-cr^{-1-\alpha}V(r)
$ by  \eqref{ww},  we only have to prove the estimate
  \begin{equation}\label{estimate1}
\left|\left(r\frac{d}{d r}\right)^j  G(r)\right|\leqq C<\infty , \ \
j=1,2,\ldots,n-1, \quad \textrm{for} \  \ 0<r<\ve,
\end{equation}
where
  $G(r)=r^{-\alpha}\int_{S^{n-1}} \sin^\ell (r \sigma_1)\;d \sigma.$
  We represent $G(r)$ as
  \[
G(r)=r^{\ell-\alpha}F(r), \ \quad \ F(r): =\int_{S^{n-1}}  s(r
\sigma_1)\sg_1^\ell \;d \sigma
\]
  where
$s(t)=\left(\frac{\sin t}{t}\right)^\ell$ is an analytic function and
therefore $F(r)$ is an analytic function in $r$. Then estimate
\eqref{estimate1} becomes obvious since $\ell-\al>0$.
\end{proof}

\subsubsection{Proof for the case of the multiplier  $\mathcal{B}_3(r)$}\label{B3}

$\left.\right.$

\vs{3mm}
 As mentioned above, we treat separately the cases $n=1$ and $n\geqq 2$.

\vs{2mm}In the case $n=1$ we just have to show that $\mathcal{B}(r)$ and
$r\mathcal{B}^\prime(r)$ are bounded on $[0,\infty].$ The boundedness of
$\mathcal{B}_3(r)$ is evident on any subinterval $(N,N_1), N_1>N$ and it
suffices to note that there exist the finite value  $\mathcal{B}(\infty)$,
see the proof of Lemma \ref{nonvanish}. To show that
$r\mathcal{B}_3^\prime(r)$ is bounded, it suffices to check that $r[r^\al
w(r)]^\prime$ is bounded for large $r$. From \eqref{ww} we have
$$r[r^\al
w(r)]^\prime= r^\al w(r)-c \sin^\ell r,$$
which is bounded.

\vspace{3mm}

We pass now to the case $n\geqq 2.$
\begin{lemma}\label{mainlemmaas} Let $n\geqq 2$.
The kernel $a_3(r)$ is vanishing at infinity faster than any power and admits
the estimate:
\begin{equation}\label{rapidly}
|a_3(r)|\leqq \frac{C}{r^{\frac{n-1}{2}}(1+r)^m}, \ \quad 0<r<\infty,
\end{equation}
where $m=1,2,3, \ldots$ is arbitrarily large, and $C=C(m)$ does not depend on
$r$.
\end{lemma}
\begin{proof}
\textit{1). \  Estimation as $r\to 0$}.
 By the Fourier inversion formula for radial functions we have
\begin{equation}\label{0}
a_3(r)= \frac{(2\pi)^\nu}{r^{\nu-1}}\intl_0^\infty t^\nu
J_{\nu-1}(rt)[\mathcal{B}_3(t)-\mathcal{B}_3(\infty)]\, dt, \ \quad
\nu=\frac{n}{2}.
\end{equation}
 From \eqref{0}  we have
\begin{multline}\label{rrr}
|a_3(r)|\leqq  \frac{(2\pi)^\nu}{r^{\nu-1}}\intl_{N-\dl}^N \!\! t^\nu
|J_{\nu-1}(rt)\mathcal{B}_3(t)|\,d t \;+ \\
 + \left|\frac{(2\pi)^\nu}{r^{\nu-1}}\intl_N^\infty \!\! t^\nu
J_{\nu-1}(rt)\mathcal B (t)\,d t\right|.
\end{multline}
We make use of the asymptotics obtained in \eqref{pro} and get

\begin{multline}\label{rrrr}
|a_3(r)| \leqq \frac{c}{r^{\nu-1/2}} + \sum\limits_{i=0}^{\frac{\ell}{2}-1}
\frac{c_i}{r^{\nu-1}} \left|\intl_N^\infty
J_{\nu-1}(rt)J_{\nu-2}(\ell_it)\,d t\right| + \\+
\frac{c}{r^{\nu-1}} \intl_N^\infty  t^\nu |J_{\nu-1}(rt)|\,\frac{d
t}{t^{\nu+\frac{3}{2}}},
\end{multline}
where $c_i$ are constants.  Since $\displaystyle |J_{\nu-1}(t)|\leqq
\frac{ct^{\nu-1}}{(t+1)^{\nu-1+\frac{1}{2}}}$, the last term is easily
estimated:
\[
\frac{c}{r^{\nu-1}} \intl_N^\infty t^\nu |J_{\nu-1}(rt)|\,\frac{d
t}{t^{\nu+\frac{3}{2}}} \leqq
\frac{c}{r^{\nu-\frac{1}{2}}}\intl_N^\infty\frac{t^{\nu-1}}{\left(t+
\frac{1}{r}\right)^{\nu-\frac{1}{2}}}\frac{d t}{t^\frac{3}{2}} \leqq
\frac{c}{r^{\nu-\frac{1}{2}}}.
\]
Thus
$$|a_3(r)| \leqq \frac{c}{r^{\nu-\frac{1}{2}}} +
\sum\limits_{i=0}^{\frac{\ell}{2}-1} \frac{c_i}{r^{\nu-1}}
\left|\intl_N^\infty J_{\nu-1}(rt)J_{\nu-2}(\ell_it)\,d t\right|
$$ as $r\to 0$. It is known that the integral $\intl_0^\infty
J_{\nu-1}(rt)J_{\nu-2}(\ell_it)\;d t$ converges when $n\geqq 2$; it is equal
to zero, if $n>2$ and $-\frac{1}{\ell_i}$, if $n=2$, see \cite[formula
6.512.3]{graryz}  (use also the fact that $J_{\nu-2}(r)=J_{-1}(r)=-J_1(r)$ if
$n=2$). Then

$$|a_3(r)| \leqq \frac{c}{r^{\nu-\frac{1}{2}}}+
\sum\limits_{i=0}^{\frac{\ell}{2}-1} \frac{c_i}{r^{\nu-1}}\left|\intl_0^N
J_{\nu-1}(rt)J_{\nu-2} (\ell_it)\,d t\right|\leqq
\frac{c}{r^{\nu-\frac{1}{2}}}$$ which proves \eqref{rapidly} as $r\to 0.$

\vs{3mm}\textit{2). Estimation as $r\to\infty$.} Since the integral in
\eqref{0} is not absolutely convergent for large $t$, it is not easy to treat
the case $r\to\infty$ starting from the representation \eqref{0}.  So we
transform this representation. We interpret the  integral in \eqref{0} in the
sense of regularization:
\begin{equation}\label{y0}
a_3(r)= \lim_{\ve\to 0}\frac{(2\pi)^\nu}{r^{\nu-1}}\intl_0^\infty \me^{-\ve
t}t^\nu J_{\nu-1}(rt)[\mathcal{B}_3(t)-\mathcal{B}_3(\infty)]\,d t.
\end{equation}
and  before to pass to the limit in \eqref{y0},  apply formula \eqref{3} with
$f(t)=\me^{-\ve t}[\mathcal{B}_3(t)-\mathcal{B}_3(\infty)].$ Then conditions
\eqref{2} are satisfied  so that formula \eqref{3} is applicable and after
easy passage to the limit we obtain
\begin{equation}\label{a(r)}
a_3(r)=\frac{(-1)^m(2\pi)^\nu}{r^{\nu+m-1}}
\sum\limits_{k=1}^mc_{m,k}\intl_0^\infty t^{\nu+k-m}
J_{\nu+m-1}(rt)\mathcal{B}_3^{(k)}(t)\,d t
\end{equation}
for every $m\geqq 1.$ The last representation already allows to obtain the
 estimation as $r\to\infty$. From \eqref{a(r)} we get
\begin{multline}\label{a(r)appl}
|a_3(r)|\leqq  \frac{c}{r^{\nu+m-1}}\intl_{N-\dl}^N t^{\nu}
|J_{\nu+m-1}(rt)\mathcal B^{(m)}(t)|\,d t+\\ +
\frac{c}{r^{\nu+m-1}}\left|\intl_N^\infty t^{\nu}
J_{\nu+m-1}(rt)\mathcal B^{(m)}(t)\,d t\right|+\\
+\sum\limits_{k=1}^{m-1}\frac{c_k}{r^{\nu+m-1}}\intl_N^\infty t^{\nu+k-m}
\left|J_{\nu+m-1}(rt)\mathcal B^{(k)}(t)\right|\,d t.
\end{multline}
The function $\mathcal B^{(m)}(t)$ is bounded on $[N-\dl,N]$ so that the
estimation of the first term is obvious.  Since $|J_{\nu+m-1}(rt)|\leqq
\frac{c}{\sqrt{rt}}$ and $|\mathcal B^{(k)}(t)|\leqq c t^{-\xi}$,
$\xi=\min\{k,\nu+1/2\}$, see \eqref{pronew}, the last sum in \eqref{a(r)appl}
is estimated by $\frac{c}{r^{\nu+m-\frac{3}{2}}}$.

It remains to estimate the second term. We make use of the asymptotics in
\eqref{pronew} again and obtain
\begin{equation}\label{final} |a_3(r)|\leqq
\frac{1}{r^{\nu+m-\frac{3}{2}}} \left(c+ \sum\limits_{i=1}^
{\frac{\ell}{2}-1}c_i\left|\intl_N^\infty J_{\nu+m-1}(rt)J_{\nu-2}(\ell_i
t)\, d t\right|\right).
\end{equation} It is known that the last integral
converges when $\nu+\frac{m}{2}>1$ and
$$\intl_0^\infty J_{\nu+m-1}(rt)J_{\nu-2}(\ell_i
t)\, d t= \frac{\gm}{r^{\nu-1}}, \ \quad r>\ell_i,$$ where $\gm$ is a
constant $\left(\gm=
\ell_i^{\nu-2}\frac{\Gm\left(\nu-1+\frac{m}{2}\right)}{\Gm\left(\nu-1\right)
\Gm\left(1+\frac{m}{2}\right)}\right)$, see \cite[formula 6.512.1]{graryz}.
Then from \eqref{final} we get \eqref{rapidly}.
\end{proof}

\subsection{Proof of Theorem \ref{main}}\label{smain}

\begin{proof} Assume that the limit $\lim_{\varepsilon \to 0+} \mathbb
D^\alpha_\varepsilon f$ in \eqref{principal} exists. We  express
$\frac{1}{\ve^\al}(I-P_\varepsilon)^\alpha f $ via
$\varphi_\varepsilon(x):=\mathbb D^\alpha_\varepsilon f(x)$ in
``averaging" terms:
\begin{equation}\label{5.1}
\frac{1}{\varepsilon^\alpha}(I-P_\varepsilon)^\alpha
f(x)=c\varphi_\varepsilon(x)+\frac{1}{\varepsilon^n} \int_{\mathbb R^n}
a\left ( \frac{x-y}{\varepsilon} \right) \varphi_\varepsilon(y)\;d y
\end{equation}
where $a (x)\in L^1(\mathbb R^n)$ and $a(x)$ is the inverse
Fourier transform of the function $A(x)-A(\infty),\; c=A(\infty)$
and
\begin{equation}\label{uu}
  c+\int_{\mathbb R^n} a(y)\;d y=1.
\end{equation}

Representation \eqref{5.1}-\eqref{uu}, verified via Fourier
transforms:
\begin{equation}\label{5.4}
   F \left(\frac{1}{\varepsilon^\alpha}(I-P_\varepsilon)^\alpha
f \right)(x)= A(\varepsilon x) F(\mathbb D^\alpha_\varepsilon
f)(x)
\end{equation}
was given in \cite{hypsam}  for the case of constant $p$ and thus
valid  for $f \in C^\infty_0(\mathbb R^n)$. Then \eqref{5.1} holds
for $f\in L^{p(\cdot)}(\mathbb{R}^n)$ by the continuity of the
operators on the left-hand and right-hand sides in
$L^{p(\cdot)}(\mathbb{R}^n)$; for the left-hand side see
\eqref{2.6}, while the boundedness of the convolution operator on
the right-hand side follows from the fact that the Fourier
transform of its kernel is a Fourier $p(\cdot)$-multiplier by
Lemma \ref{amultiplier}.

With  $\varphi=\mathbb D^\alpha f =\lim_{\ve\to 0}\varphi_\ve$,
from \eqref{5.4} we have
\begin{equation}\label{new}
   \left \| \frac{1}{\varepsilon^\alpha}(I-P_\varepsilon)^\alpha
f -\varphi  \right \|_{p(\cdot)} \leqq C  \| \varphi_\varepsilon-
\varphi \|_{p(\cdot)} + \| T_{\ve} \varphi -\varphi \|_{p(\cdot)}
\end{equation}
where $T_\ve$ is the operator defined by the right hand side of
\eqref{5.1}, that is, the operator generated by the Fourier
 multiplier $A(\ve x)$. The right-hand side of \eqref{new} tends
 to zero, the first term by the definition of $\vi$, the second
 one by
  Lemma \ref{12}, since  $A$ satisfies the Mikhlin condition as
  proved in
  Lemma \ref{amultiplier}.

Suppose now that $\lim_{\varepsilon\to 0+}
\frac{1}{\varepsilon^\alpha}(I-P_\varepsilon)^\alpha f$ exists in
$L^{p(\cdot)}(\mathbb{R}^n)$. By \eqref{5.4} we have
\begin{equation}\label{obratno}
 F(\mathbb D^\alpha_\varepsilon f)(x)=B(\varepsilon x)F \left(
\frac{(I-P_\varepsilon)^\alpha f}{\varepsilon^\alpha}\right)(x)
\end{equation}
for $f\in C^\infty_0(\mathbb R^n)$, where $B(x)=1/A(x)$. Since
$B(x)$ is also a Fourier multiplier by Theorem \ref{amultiplier}, the
arguments are the same as in above the passage from $\lim_{\ve\to
0}\mathbb{D}_\ve^\al f$ to $\lim_{\varepsilon\to 0+}
\frac{1}{\varepsilon^\alpha}(I-P_\varepsilon)^\alpha f.$
\end{proof}

\subsection{Proof of Theorem \ref{main2}}

 $\left.\right.$

\vs{2mm} The ``only if" part of Theorem \ref{main2} is a consequence of
Theorem \ref{main}.

To prove the ``if" part, suppose that \eqref{vdnh} holds. From \eqref{obratno}
we obtain that
\[
\|\mathbb{D}_\ve^\al f\|_{p(\cdot)} \leqq C
\left\|\frac{1}{\ve^\al}(I-P_\ve)^\al f \right \|_{p(\cdot)}\leqq C,
\]
 since $A(\ve x)$ is
a uniform Fourier multiplier in $L^{p(\cdot)}(\rn)$ by Theorem \ref{main1}
and Lemma \ref{dilation}. To finish the proof,  it remains to  refer to
Theorems \ref{rieszrange1} and \ref{alm}.

\subsection{Proof of Theorem \ref{main1}}\label{smain1}

$\left.\right.$

\vs{3mm} Theorem \ref{main1} is an immediate consequence of Theorem
\ref{main} and  Propositions \ref{besselcharacterization} and \ref{alm}.

\section{Appendix}

The following recurrence formula for the $k$-th derivative of the quotient is valid
(see, e.g., \cite{xenophontos,gerrish})
\begin{equation}\label{quo}
\left( \frac{u}{v}\right)^{(k)}=\frac{1}{v}\left(u^{(k)}-k!\sum_{j=1}^k
\frac{v^{(k+1-j)}}{(k+1-j)!}\frac{\left(\frac{u}{v}\right)^{(j-1)}}{(j-1)!} \right).
\end{equation}

 By means of this formula, by induction it is not hard to check the validity of the
 following formula for the $k$-th derivative
of the fraction $\frac{1}{v(x)}, x\in \mathbb{R}^1$:
\begin{equation} \label{nth}
\frac{d^k}{d x^k}\left(\frac{1}{v}\right)=\frac{v^{(k)}}{v^2}+
\sum_{j=1}^{k-1}\frac{A_j(D)v}{v^{j+1}} +(-1)^kk!\frac{(v^\prime)^k}{v^{k+1}}
\end{equation}
where the differential operators $A_{j,k}(D)$ of order $ k$ have
the form
$$A_{j,k}(D)v= a_j\left[\frac{d^{m_j} v}{d x^{m_j}}\right]^{\al_j} \left[\frac{d^{n_j}v}
{d x^{n_j}}\right]^{\bt_j} + b_j\left[\frac{d^{p_j}v}{d
x^{p_j}}\right]^{\gm_j} \left[\frac{d^{q_j}v} {d x^{q_j}}\right]^{\dl_j}$$
where $a_j$ and $b_j$ are constants, $m_j,n_j,p_j,q_j, \al_j,\bt_j,\gm_j,
\dl_j$ are integers in $[1,k-1]$ such that
\begin{equation}\label{ultima}
m_j\al_j+n_j\bt_j = p_j\gm_j+q_j\dl_j= k.
\end{equation}

\end{document}